\documentclass[amscd,amssymb,12pt]{amsart}  

\usepackage{graphicx}
\usepackage{tikz}
\usepackage{lipsum}
\usepackage{amsthm}
\usepackage{amsmath}
\usepackage{amssymb}
\usepackage{mathtools}
\usepackage{enumitem}
\usepackage{blkarray}
\usepackage[utf8]{inputenc}

\newcommand{\D}[1]{{\mathbb#1}} 

\theoremstyle{plain}

\newtheorem{lem}{Lemma}[section]

\newtheorem{conj}{Conjecture}[section]

\theoremstyle{definition}
\newtheorem{defin}{Definition}[section]

\theoremstyle{definition}
\newtheorem{example}{Example}[section]

\theoremstyle{remark}
\newtheorem{rem}{Remark}[section]

\usepackage{fancyhdr}
\title{Periods and Applications} 
\author{Lucian M. Ionescu}
\author{Richard Sumitro}
\address{Mathematics Department, Illinois State University\\
Normal IL 61704-4520, USA}
\email{lmiones@ilstu.edu, rsumitr@ilstu.edu} 
\date{August 2017} 

\begin{document}

\begin{abstract}
Periods are numbers represented as integrals of rational functions over algebraic domains. 
A survey of their elementary properties is provided. 
Examples of periods includes Feynman Integrals from Quantum Physics and 
Multiple Zeta Values from Number Theory.  

But what about finite characteristic, via the global-to-local principle?
We include some considerations regarding periods and Jacobi sums,  
the analog of Veneziano amplitudes in String Theory.
\end{abstract}


\maketitle
\setcounter{tocdepth}{3} 
\tableofcontents

\section{Introduction}
Periods form a new and important class of numbers extending the traditional class of algebraic numbers.
They appeared in mathematicians computations, notably in \cite{Zagier-MZV}, 
and their importance was immediately recognized to be properly introduced by Kontsevich \cite{K97,K99}, 
and later studied together with Zagier \cite{KZ}\footnote{We apologize if we got some priorities wrong.}.
A quick example is $\pi$, the area of the unit disk; a conjectural non-example is $e$.

Periods are ubiquitous in mathematics, showing up as Feynman Integrals in quantum physics, 
as well as values of Multiple Zeta Functions (MZV) in Number Theory.
There is a fast growing interest to understand and organize this new area of study 
\cite{Yves,Brown,Schnetz,Waldschmidt} etc..

In this article we review the main properties of periods, with examples coming from calculus, 
as volumes, from physics as Feynman integrals,
and from number theory, as multiple zeta values. 
Identities between periods are non-trivial to prove, and Kontsevich's conjecture is recalled.  
A few proofs of admissible identities are provided, or attempted. 

Finally, a new point as far as we know, the we attempt to relate periods with Jacobi sums, 
which are the finite characteristic analog of amplitudes of string interactions.
Recall that these so called Veneziano amplitudes from String Theory,
correspondent to the Feynman Amplitudes of particle interactions in Quantum Field Theory.

The main goal in this paper is to make the new subject of periods more widely known, especially amongst graduate students.
We thus avoid advanced topics related to motives \cite{Yves}, which we ourselves hardly understand!
A second objective is to emphasize the connection between (quantum physics and mathematics, 
especially Number Theory, which should not be surprising when reminded that {\em quantum} means {\em discrete}!

The article is organized as follows.
We start by presenting periods, with examples and stating their elementary properties, 
following Zagier and Kontsevich \cite{KZ}.
Next, we survey two instances of periods, 
showing the tight connection between Physics and Mathematics, 
especially with Analytic Number Theory:
Feynman integrals and their yet to be understood connection with Multiple Zeta Values \cite{Schnetz,QuantaMagazine},
and second, Veneziano amplitudes, which are the String Theory analog of the former, 
and Jacobi sums, with deep roots in Algebraic Number Theory / Geometry,
awaiting to be further explored \cite{LI:IHES-Project,LI:P-FI-JS}.
In conclusion, we summarize the main aspects presented, and plan for further developments.

\vspace{.1in}
This article is based on the research project of the second author, 
focusing on periods and Feynman integrals \cite{RS:ISU-report}, 
with further contributions from the second author regarding the second application.

\section{What are {\em periods?}}
Recall the ``traditional'' number systems $\{1,2,3, many\}$ (Neanderthal man?), $N\subset Z\subset Q\subset R \subset C$,
with their ``raison d'etre'' (motivation): to extend the number system allowing to perform arithmetic operations or to solve algebraic equations. The ``real numbers'', as an ``overkill'' in this direction, introduced a lot of ``junk'' (unrealistic numbers like 
Liouville's constant etc. \cite{Real-fish}). 
The algebraic numbers, i.e. solutions of polynomial equations with rational coefficients, are (reasonably) well understood by now, 
in spite of Pythagorean's denial of their existence. 
We also got used with the {\em transcendental numbers}, i. e. irrational numbers, but not algebraic, 
with infinitely many ``random'' digits in their decimal representations (never-mind Cauchy here!), 
without questioning their meaning (``They mathematically exist, therefore we don't doubt!'', a sort of Descartes' converse ...). 
But not all transcendental numbers are created equal! The ubiquitous $\pi$, more of a geometric symbol or trademark,
stands above all ... as well as $e\approx 2.1417...$, reigning in the ``Analysis Kingdom''.
Are there more?
Note that Cantor has proven that $\bar{\mathbb{Q}}$ is countable and transcendental numbers are uncountable, 
so there should be plenty ... but are they ''useful''?  And how different are they, what other ``types'' of numbers we 
may want to distinguish?
\begin{rem}
Regarding the amount of information contained in such a transcendental number,
a ``generic'' one contains and infinitely amount of information in its decimal representation, and we cannot provide it; 
our examples make use of a finite description, e.g. for Liouville's number, $e$, $\pi$ etc.
Entropy (amount of info), could be a criterion for the complexity of such a transcendental number,
yet in this article, a more algebraic-geometric aspect is emphasized.
\end{rem}

Recently, one more important class of numbers {\em in Mathematics} was identified in a systematic and coherent way 
\cite{Zagier-MZV,K97,K99,KZ}\footnote{Physicists probably new their importance before, from Physics!}, 
which lay in between $\bar{\mathbb{Q}}$ and $\mathbb{C}$, called \textbf{periods}. 
These numbers appear as rational integrals, and form a countable class, extending the class of algebraic numbers. 
Moreover, the periods occurring in applications as such integrals, are usually transcendental numbers,
and not algebraic.
They contain only a finite amount of information, 
and it was conjectured that they can be identified in an algorithmic way. 
Periods appear surprisingly often in various formulas and conjectures in mathematics, 
and often provide a {\em bridge between problems coming from different disciplines}.  

\subsection{Definition and Properties}\label{S:Examples}
\begin{defin}\label{D:Period}
A {\em period} $P$ is a complex number whose real and imaginary parts are values of 
absolutely convergent integrals of {\em rational functions} with rational coefficients, 
over {\em rational domains} in $\mathbb{R}^n$, i.e. given by polynomial inequalities with rational coefficients.
\end{defin}
Notice that by definition, since we only consider rational functions to define $P$, we see that $P$ is countable since rationals are countable. 

Ex. $\pi$ is a period. Several possible {\em representations} as an integral follow:
\begin{eqnarray*}
\pi = \int \int_{x^2+y^2\leq 1} 1 dxdy = 2\int_{-1}^1 \sqrt{1-x^2}dx = \int_{-1}^1\frac{1}{\sqrt{1-x^2}}dx = \int_{-\infty}^{\infty} \frac{1}{1+x^2}dx
\end{eqnarray*}

Ex. $\sqrt{2}$ is a period:
\begin{eqnarray*}
\sqrt{2}= \int_{2x^2\leq 1} 1 dx
\end{eqnarray*}

Ex. Any logarithm of an integer, such as $\log(3)$ is a period
\begin{eqnarray*}
\log(3) = \int_1^3 \frac{1}{x}dx 
\end{eqnarray*}

Ex. For all natural numbers $k>1$, the values of the \textit{Riemann zeta function}
\begin{eqnarray*}
\zeta(k) = \Sigma_{n\geq 1} \frac{1}{n^k}
\end{eqnarray*}
are periods (see the following Lemma), 
as are also the {\em multiple zeta values} (MZV) at integer arguments $k_i \in \mathbb{N}, k_i\geq 2$:
\begin{eqnarray*}
\zeta(k_1,\cdots ,k_l) = \Sigma_{0<n_1<\cdots < n_l} \frac{1}{n_1^{k_1}\cdots n_l^{k_l}}.
\end{eqnarray*}

\begin{lem}
For $k \in \D{N}, k \geq 2$, the number 
\begin{eqnarray*}
\zeta (k) = \Sigma_{n \geq 1} \frac{1}{n^k}
\end{eqnarray*}
is a period.
\end{lem}
\begin{proof}
Note that 
\begin{eqnarray*}
\zeta(k) = \int_{1>t_1>\cdots > t_k > 0} \frac{dt_1}{t_1}\cdots \frac{dt_{k-1}}{t_{k-1}}\frac{dt_k}{1-t_k}
\end{eqnarray*} 
\end{proof}

For example, when $k=2$:
\begin{eqnarray*}
\zeta(2) = \int\limits_{0<x<y<1} \frac{dx}{1-x}\frac{dy}{y}
\end{eqnarray*}
The arithmetic nature of the values of the Riemann zeta function at positive even integers is known since Euler:
\begin{eqnarray*}
\pi^{-2k}\zeta(2k) \in \mathbb{Q} 
\end{eqnarray*}
for $k \geq 1$. 
These rational numbers involve Bernoulli numbers. 
Since Lindermann's theorem asserts that the number $\pi$ is trascendental, so is $\zeta(2k)$ for any integer $k \geq 1$.  

A more complex example of period is provided by the logarithmic Mahler measure of a Laurent polynomial \cite{KZ}, p.5:
$$\mu(P)=\idotsint\limits_{|x_1=1| , ... , |x_n|=1}  \log |P(x_1,...,x_n)|   dx_1 \dots dx_n.$$

\vspace{.2in}
But what about other ``interesting'' transcendental numbers, which are {\em not} periods?
Interestingly, $e$ is transcendental, but conjecturally not a period, unlike $\pi$ which is both. 
Another example of non-period numbers is the {\em Liouville number} $L$, which is an irrational number, since it has the 
rational approximation property that for every positive integer $n$, there exist integers $p$ and $q$ with $q>1$ and such that 
\begin{equation*}
0<\vert x -\frac{p}{q} \vert < \frac{1}{q^n}
\end{equation*}  
but probably not a period.
Another famous ``mathematical constant'', the Euler-Mascheroni constant:
$$\gamma = \lim_{n \rightarrow \infty} (1+\frac{1}{2}+\cdots + \frac{1}{n}-log(n))=
\int_1^\infty \left(  \frac1{[ x ]} -\frac1{x} dx   \right)$$ 
is (conjecturally) not a period, although it ``looks'' like one, if it weren't for the 
arithmetic {\em floor} function $[x]$.

\begin{rem}
The ``problem'' in this last example, is ``mixing'' algebraic functions and domain, with a step function, 
but which is still algebraic: the ``other projection'', from $R$ to $Z$,
of the arithmetic projection modulo $1$: $Z\to R\to R/Z$\footnote{... reminiscent of a chain homotopy: $[x]+fr(x)=x$.}.
Are these more ``general'' periods?
\end{rem}

Note again, that these numbers are only {\em conjecturally} classified as not periods;
there may be a integral that can represent such numbers to classify as one, but not found yet.


\subsection{Identities between periods}
On one hand the set $P$ is countable and each element of it can be described by a finite amount of information (namely, the integrand and domain of integration defining the period). 
On the other hand, a \textit{priori}, there are many ways to write a complex number as integrals, which yield the same value, the period.
The problem is exacerbated by the fact that two different periods may be numerically close and yet be distinct. 
For example, consider the following two periods \cite{KZ}, p.6: 
$$ \frac{\pi \sqrt{163}}{3} \qquad \text{and} \qquad \log(640320),$$
in which both have decimal expansions beginning 13.36972333037750... 
(computed by SAGE\footnote{How can one find such examples!? ... generating a ``lattice'', using $\pi$ and $\log$, maybe?}). 
Even more amazingly, the two periods 
$$\frac{\pi}{6}\sqrt{3502}\qquad \text{and} \qquad  \log(2\prod_{j=1}^4(x_j+\sqrt{x_j^2-1})),$$
where $x_1 = \frac{1071}{2}+92\sqrt{34}, x_2 = \frac{1553}{2}+133\sqrt{34}, x_3 =429+304\sqrt{2}, x_4 = \frac{627}{2}+221\sqrt{2}$, which agree numerically to more than 80 decimal digits and nevertheless are different\footnote{... dito! See \cite{Shanks-23}.}.

For {\em algebraic numbers}, there may also be apparently different expressions for the {\em same number}, 
such as \cite{Shanks-22}
\begin{equation*}
\sqrt{11+2\sqrt{29}}+\sqrt{16-2\sqrt{29}+2\sqrt{55}-10\sqrt{29}} = \sqrt{5}+\sqrt{22}+2\sqrt{5},
\end{equation*} 
but we can check their equality easily, either by finding some polynomial satisfied by each number and computing the g.c.d. of these polynomials or else by calculating both numbers numerically to sufficiently high precision and using the fact that 
{\em two different solutions of algebraic equations with integer coefficients of given degree and height \cite{Silverman-height},
cannot be too close to each other}. 

Now the question is: ``Can we do something similar for periods?''. 
From calculus, we have several transformation rules between integrals, i.e. ways to prove identities between integrals. 
For integrals of functions in one variable there rules are as follows:
%
\begin{enumerate}

\item {\bf Additivity} (Domain and integrand):
\begin{eqnarray*}
\int_a^b(f(x)+g(x))dx &=& \int_a^b f(x)dx + \int_a^b g(x)dx\\
\int_a^b f(x)dx &=& \int_a^c f(x)dx + \int_c^b f(x)dx
\end{eqnarray*}

\item {\bf Change of variables}: 
if $y=f(x)$ is an invertible change of variables, then 
\begin{eqnarray*}
\int_{f(a)}^{f(b)} F(y) dy = \int F(f(x))f'(x)dx
\end{eqnarray*}

\item {\bf Newton-Leibniz formula}:
\begin{eqnarray*}
\int_a^b f'(x)dx = f(b) - f(a)
\end{eqnarray*}

\end{enumerate}
In the case of multi-dimensional integrals, one puts the Jacobian of an invertible change of coordinates in Rule 2,
and replaces the Newton-Leibniz formula by Stoke's formula in Rule 3. 

\subsection{Kontsevich's Conjecture}
Now, are these rules ``enough''? 
From observation (and much more: \cite{KZ}, p.7), there is a widely-held belief which can be stated as follows:
\begin{conj}
If a period has two integral representations, then one can pass from one formula to another using only rules (1)-(3) in which all functions and domains of integration are algebraic with coefficients in $\bar{\mathbb{Q}} $.
\end{conj}
In other words, {\em there is no coincidence} 
\footnote{A. Einstein: ``Coincidence is God's way of remaining anonymous''.},
that two integrals of algebraic functions are equal
 and such equalities can be proven by utilizing only rule (1)-(3). 
 Whenever we have expressed the identity to be proved as an equality between two periods
 {\em and assume that Conjecture holds true}, 
 we call such equality as \textbf{accessible identity}


\subsection{Examples of Identities}
Examples of such identities involving the zeta values are due to Euler.
\begin{example} 
Consider the following equality, proved by Euler in 1734: 
$$\zeta(2) = \frac{\pi^2}{6}.$$ 
Since $\zeta(2)$ and $\pi$ are periods, we see that given equality is an accessible identity. 
But is there a proof using {\em only} Rules 1-3?
Such a proof was given by Calabi \cite{Calabi}, and adapted in \cite{KZ}
(we provide some additional details).
\end{example}
\begin{proof}
Set 
\begin{equation*}
    I = \int_0^1 \int_0^1 \frac{1}{1-xy} \frac{\,dx \,dy}{\sqrt{xy}} 
\end{equation*}
Expanding $1/(1-xy)$ as a geometric series and integrating by terms, we find that 
\begin{equation*}
I = \Sigma_{n=0}^{\infty} (n+\frac{1}{2})^{-2} = (4-1)\zeta(2)
\end{equation*}
in which we found a period representation of $\zeta(2)$. Now, making the change of variables, 
\begin{eqnarray*}
x &=& \eta^2 \frac{1+\xi^2}{1+\eta^2}\\
y &=& \xi^2 \frac{1+\eta^2}{1+\xi^2}
\end{eqnarray*}
with using Jacobian $\vert \frac{d(x,y)}{d(\eta, \xi} \vert = \frac{4\eta\xi(1-\eta^2\xi^2)}{(1+\eta^2)(1+\xi^2)} = 4\frac{(1-xy)\sqrt{xy}}{(1+\eta^2)(1+\xi^2)}$, we find that 
\begin{equation}
I = 4 \iint_{\eta, \xi > 0, \, \eta \xi \leq 1} \frac{\,d\eta}{1+\eta^2}\frac{\,d\xi}{1+\xi^2} 
\end{equation}
Now consider the involution $w: (\eta, \xi) \mapsto (\eta^{-1}, \xi^{-1})$. Note that $\eta \xi \leq 1$ implies that $ \eta^{-1} \xi^{-1} >1$, so that together the union of the two regions is, essentially, the first quadrant. Then 
\begin{eqnarray*}
2 \iint_{\eta, \xi > 0, \, \eta \xi \leq 1} \frac{\,d\eta}{1+\eta^2}\frac{\,d\xi}{1+\xi^2} &=& \iint_{\eta, \xi > 0, \, \eta \xi \leq 1} \frac{\,d\eta}{1+\eta^2}\frac{\,d\xi}{1+\xi^2}+\iint_{\eta, \xi > 0, \, \eta \xi > 1} \frac{\,d\eta}{1+\eta^2}\frac{\,d\xi}{1+\xi^2}
\\&=&\iint_{\eta, \xi > 0, \, \eta \xi \leq 1} \frac{\,d\eta}{1+\eta^2}\frac{\,d\xi}{1+\xi^2} +\iint_{\eta, \xi > 0, \, \eta \xi > 1} \frac{\,d\eta}{1+\eta^2}\frac{\,d\xi}{1+\xi^2}
\\
&=& \iint_{\mathbb{R}^+\times \mathbb{R}^+} \frac{\,d\eta}{1+\eta^2}\frac{\,d\xi}{1+\xi^2}
\end{eqnarray*}
So, by utilizing Fubini's theorem, we see that 
\begin{equation*}
2 \iint_{\eta, \xi > 0, \, \eta \xi \leq 1} \frac{\,d\eta}{1+\eta^2}\frac{\,d\xi}{1+\xi^2} = \int_0^{\infty}\frac{\,d\eta}{1+\eta^2}\int_0^{\infty} \frac{\,d\xi}{1+\xi^2} 
\end{equation*}
Thus, from Eq (1), we see that
\begin{equation*}
   I = 2 \int_0^{\infty}\frac{\,d\eta}{1+\eta^2}\int_0^{\infty} \frac{\,d\xi}{1+\xi^2} 
\end{equation*}

and consider the integral form of $\pi = \int_{-\infty}^{\infty} \frac{1}{1+x^2}\,dx$, we see that $I = \pi^2/2$.
Thus, we obtain the equality $3 \zeta(2) = I = \pi^2/2$. which proves the equality as required.
\end{proof}

As another example, we have the following accessible identity \cite{KZ}, involving the Mahler measure $\mu(P)$ from 
section \S\ref{S:Examples}:
\begin{eqnarray*}
\mu (x+y+16+\frac{1}{x}+\frac{1}{y} = \frac{11}{6}\mu (x+y+5+\frac{1}{x}+\frac{1}{y}).
\end{eqnarray*}
As an exercise (suggested in \cite{KZ}, p.9), try proving it using only the rules $1 - 3$;
the present authors tried, and ended up comparing some hypergeometric functions,
to be addressed elsewhere \cite{IS:MahlerMeasure}.


\section{Feynman Integrals and Multiple Zeta Values}
The Feynman integral has many usages in both physics and mathematical field. 
It is the modern approach to Quantum Field theory. 
More interestingly, there are many connection between periods and Feynman integrals.

\subsection{Feynman Graphs and Feynman Integrals}
Historically, Feynman integrals were mainly used in the perturbative approach to QFT,
starting from a Lagrangean formulation, with a free part controlling the motion of the fields and an interaction term,
determining the interactions between fields.

Following \cite{phi4}, we consider massless euclidean $\phi^4$-theory in 4 dimensions with 
Lagrangean interaction term normalized to 
\begin{eqnarray*}
L_{int} = \frac{16\pi^2 g}{4!}\int_{\mathbb{R}^4}d^4x\phi(x)^4,
\end{eqnarray*}
but without the physics details involved.
We focus on the 4-point-function and obtain for the scattering amplitude of a Feynman-graph $\Gamma$ (for examples, see Fig. 1). 
\begin{align}
A_{\gamma} = (2\pi)^4\delta^4(q_1+q_2+q_3+q_4)\frac{16\pi^2 g}{\vert q_1 \vert ^2 \cdots \vert q_4 \vert ^2}\cdot(\frac{g}{\pi^2})^{h_1}\int_{\mathbb{R}^{4h_1}}d^4p_1\cdots d^4p_{h_1} \frac{1}{\Pi_{i=1}^n Q_i (p,q)}
\end{align}
where the momentum-conserving 4-dimensional $\delta$-function $\delta^4$ with 'external' moments $q_1,\cdots , q_4$, 
the ``loop order'' $h_1$ giving the number of independent cycles in $\Gamma$.  

The integral diverges logarithmically for large $p_i$. 
Since for large $p_i$ the value of the external momenta becomes irrelevant we may nullify the $q_i$ 
to characterize the divergence by a mere number (if it exists) given by the projective integral
\begin{eqnarray}
P_{\Gamma} = \pi^{-2h_1}\int_{\mathbb{P}\mathbb{R}^{4h_1-1}}\frac{\Omega(p)}{\Pi_{i=1}^{n} Q_i (p,0)}
\end{eqnarray} 
Here, we have introduced the projective volume measure which is defined in $\mathbb{P}^m$ with coordinates $x_0,\cdots, x_m$ as 
\begin{eqnarray*}
\Omega(x) = \Sigma_{i=0}^{m} (-1)^idx_0 \cdots \hat{dx_i} \cdots dx_m
\end{eqnarray*}
We assume an orientation of $\mathbb{P}\mathbb{R}^{4h_1-1}$  is chosen such that $P_{\gamma}>0$. 

In the following, we consider the differential form $\Omega(p)/\Pi Q_i(p,0)$ in Eq(3) as degree 0 meromorphic $4h_1-1$ form in complex projective space $\mathbb{PC}^{4h_1=1}$. It is of top degree as meromorphic form and hence closed in the complement of $\Pi Q_i(p,0)=0$. As odd dimensional real projective space the domain of integration is orientable and compact without boundary and thus a cycle in $\mathbb{PC}^{4h_1-1}$. However, the cycle of integration meets the singularities of the differential form which in general leads to an ill-defined integral. To ensure that the integral converges we need an extra condition on the graph $\Gamma$. 
We define a graph $\Gamma$ is \textbf{primitive} if it has $n(\Gamma) = 2h_1(\Gamma)$ edges and every proper subgraph $\gamma < \Gamma$ has $n(\gamma) > 2h_1(\gamma)$. Amazingly, we see that the period $P_{\Gamma}$, Eq (3), is well-defined if and only if $\Gamma$ is primitive. A proof is given in Proposition 5.2 from \cite{Bloch}.

\vspace{.1in}
In this sense, $\pi^{2h_1}P_{\Gamma}$ becomes a period in $\lbrace p \in \mathbb{PC}^{4h_1-1} : \Pi Q_i(p,0)\neq 0\rbrace$. 
In any case, this parametric representation of the Feynman integral 
makes $P_{\Gamma}$ an algebraic period in the sense of Kontsevitch and Zagier. 

One finds these periods in all sorts of perturbative calculations (like the beta-function or the anomalous dimension) within the QFT considered. In fact, the role that periods play in Hopf-algebra of renormalization \cite{Connes}, 
suggests that there might exist a clever renormalization scheme such that they form a complete 
$\mathbb{Q}$-base for the coefficients of the perturbative expansion of scalar functions. 
This gives periods a prominent role within QFT.   

\begin{example}
Consider the two Feynman graphs from Fig.1:
\begin{center}
\includegraphics[scale =.8]{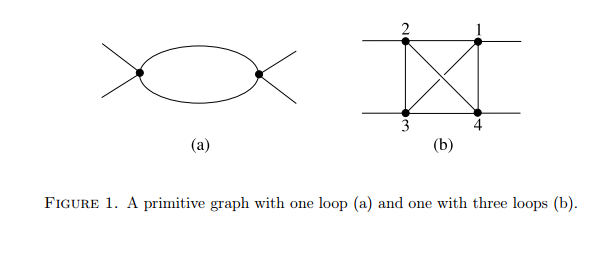}
\end{center}
Regarding the graph plotted in Fig. 1(a), the corresponding period is:
\begin{eqnarray*}
P_1 &=& \pi^{-2} \int_{\mathbb{P}\mathbb{R}^3} \frac{\Omega(p)}{\vert p \vert ^2 \vert p \vert ^2}\\
&=& \pi^{-2} \int_{\mathbb{R}^2} \frac{d^3 \textbf{p}}{(\textbf{p}^2+1)^2}\\
&=&\pi^{-2}4\pi \int_0^{\infty} \frac{p^2}{(p^2+1)^2}dp\\
&=& 1
\end{eqnarray*}
In the second line, we used $p=(1,\textbf{p})$ to make the integral affine,
and in the third line, we introduced polar coordinates to transform the integral to a standard one-dimensional integral. 
Note that graph $1 (a)$ is so far the only one evaluating to a $\phi^4$ period known to be a rational number. 
Most likely, it is the only rational $\phi^4$-period. 
\end{example}

\subsection{Feynman Rules}
To see how Mathematics enters the Physics of QFT, we will briefly sketch the algorithmic nature of Feynman integrals associated
to Feynman graphs.

{\em Feynman rules} are prescriptions allowing to translate a Feynman-graph $\Gamma$ into an analytical expression, an improper integral representing the scattering amplitude $A_{\Gamma}$. 
In our setup - primitive 4-point functions without external momenta in massless 4-dimensional $\phi^4$-theory - these expressions evaluate to positive numbers. 

There are four different ways to use Feynman rules, depending on the type of variables used: 

\quad A) {\em Position and momentum space}, where integrands are products of inverted quadrics and the variables are 4 dimensional vectors assigned to vertices and cycles, respectively; 

\quad B) Alternatively, we may use Feynman's {\em parametric space} either in its original form or in a {\em dual version}, with variables attached to edges of the graph. 

Although the transition from position or momentum space to parametric space is due to Feynman, it is known in the mathematical literature as ``Schwinger-trick''.  
To avoid confusion, we stick to this name in the following diagram that exemplifies the interconnection between the different approaches. Below are the propagators for a massless bosonic quantum field theory:
\begin{center}
\includegraphics[scale=1]{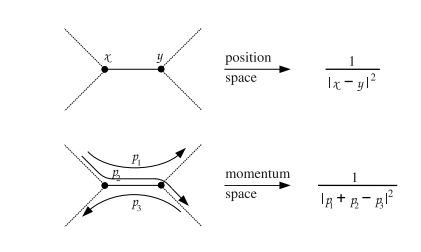}
\includegraphics[scale=1]{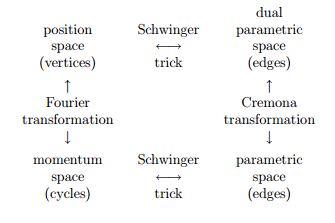}

Figure 2. Feynman rules and Schwinger's ``trick'' \cite{phi4}, p.7.
\end{center}
In position space, every edge joining vertices with variables $x,y \in \mathbb{R}^4$ contributes by a factor $1/\vert x-y \vert ^2$ to the Feynman integrand (as seen in the graph above). 
In momentum space, every edge contributes by a factor $1/ \vert \Sigma \pm p_i \vert ^2$ with variables $p_i \in \mathbb{R}^4$ associated to  a basis of cycles of the graph $P_i$, that run through the edge in one ($+$ sign) or opposite ($-$ sign) direction. 
The integration ranges over the whole real space. 

As a uniform approach, independent on the number of incoming/outgoing momenta, 
every 4-point graph in $\phi^4$-theory can be uniquely completed to a 4-regular graph 
by attaching one extra vertex to the external edges. 
For example, the completions of the graphs in Fig. 1 (a) and Fig. 1 (b) from Fig. 3 (\cite{phi4}, p.9),
yield the periods $P_1$ and $P_3$ 
in the census table \cite{phi4}, \S3, p.28.

\begin{center}
\includegraphics[scale=0.9]{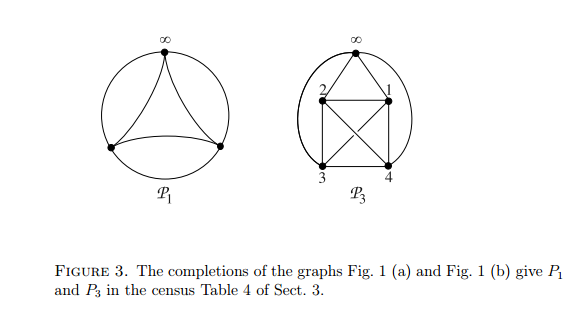} 
\end{center}


%
%
\subsection{Multiple Zeta Values}
The Riemann Zeta Function is just the ``tip of the iceberg'' of algebraic and analytic number theory.
Its generalizations are power series associated to multiplicative characters of other number systems: $Z/nZ, F_p$ and number fields,
called Dirichlet series $DS(f)=\sum f(n) n^{-s}$.
In the special case when the arithmetic function $f:\D{N}\to \D{C}$ is a  Dirichlet character, it is called an
{\em L-function}. 

Moreover, by considering multiple series with several exponents as variables, called {\em multiple zeta functions},
one obtains values, {\em multiple zeta values} (MZV), which at special integer arguments are periods.

Here we will consider briefly the case of several independent variables.
The relation to periods comes from their integral representations, using Chen iterated integrals.

The {\em multiple zeta functions} are nested series, generalizing the Riemann Zeta function:
\begin{eqnarray*}
\zeta(s_1,\cdots , s_k) = \Sigma_{n_1>\cdots > n_k > 0} \frac{1}{n_1^{s_1}\cdots n_k^{s_k}} = \Sigma_{n_1>\cdots > n_k > 0} \Pi_{i=1}^k \frac{1}{n_i^{si}}.
\end{eqnarray*}
The (multiple) series converges when $Re(s_1)+\cdots + Re(s_i)>i$ for all $i$. 
When $s_1,\cdots, s_k$ are all positive integers (with $s_1 > 1$),
these MZV are also called {\em Euler sums}. 

One of its most important properties is {\em Euler Product Formula}:
\begin{equation}
\zeta(s) = \Pi_p \frac{1}{1-p^{-s}}^{-1}
\end{equation}

As noted before, using iterated Chen integrals one can represent MZV values at integer arguments as periods.
For example, see the Euler's identity for the zeta value $\zeta(2)=\pi^2/6$.

\subsection{The First Amazing Connection!}
What is amazing, and yet not well understood, is that Feynman integrals from physics, 
which represent actual ``real'' experimental measurements, yield rational linear combinations
of such MZVs \cite{QuantaMagazine}, 
which belong to the ``purest'' part of mathematics: Number Theory!

For the reader interested in further connections with abstract algebra,
we mention that this connection is actively being studied from various points of view,
including its relations with Quantum Groups as deformations of Hopf algebras,
Drinfeld associator and KZ-equations \cite{MZV-book}.


\section{Veneziano Amplitudes and Jacobi Sums}
Periods so far were considered in characteristic zero,
as integrals over $R,C$ (continuum/``prime at infinity'' completion of rationals) or as series (discrete), over the integers.
And interestingly enough, the Math-side (MZVs) matched the Physics side, 
represented by Feynman integrals.
 
But what about the finite characteristic!?
One recurring theme in the study of Algebraic Number Theory and Algebraic Geometry 
is the relation between characteristic zero (number fields) and finite characteristic (finite fields and p-adic numbers).

At this point let us note that L-functions can be written as infinite products called Euler products.
For example the ``global'' zeta function can be written as a product
of ``local'' factors, one for each prime number $p$:
$$\zeta(s)=\prod_{p\ prime} \frac1{1-p^{-s}}.$$
The special values yielding periods, for example $\zeta(2)=\pi^2/6$, provide corresponding factorizations of periods over prime numbers.

According to a general principle ({\em Local-to-global Principle}),
the local factors provide information about the global case,
of number fields\footnote{One should include the factor corresponding to the ``prime at infinity''.}.
The ``local'' factors should have an interpretation as an ``integral'' in the corresponding characteristic. One future direction of development of the theory of periods is to investigate these local factors.

\subsection{Veneziano Amplitudes}
On the Physics side, the ``2D'' version of Feynman integrals are the Veneziano amplitudes, which have the following form:
\begin{eqnarray*}
\frac{\Gamma(\alpha) \Gamma(\beta)}{\Gamma(\alpha+\beta)}
\end{eqnarray*}
Here $\alpha$ and $\beta$ are related to the momentum of the particles involved in the scattering process, and $\Gamma$ is the gamma function. 

\subsection{Jacobi sums}
On the Math-side, Jacobi Sums can be viewed formally as a ``local'' analog of  the above Veneziano amplitude:
\begin{equation*}
  J(c,c')= \frac{g(c)g(c')}{g(c+c')}
\end{equation*}
where $c$ and $c'$ is non-trivial Dirichlet characters of the finite field $F_p=Z/pZ$ and $g(x)$ is the Gauss sum \cite{Rosen}.

It is well known that the Gamma function is in fact the ``prime at infinity'' $p=\infty$ factor of the {\em completed zeta function},
with corresponding number system $R$, the completion of $Q$ at $p$.
All these number systems are related at the level of {\em adeles}, as a duality for the rationals $Q$.
It is therefore intriguing to delve into a study of the above mentioned relationship, in order to gain a better understanding of both sides.
We will not get into details in this article, and just mention that Jacobi sums are 2-cocyles of Gauss sums, which are the finite fields analogs of Euler's Gamma function.

It is further intriguing, as an example, that the Weil zeros of the Hasse-Weil zeta function of an elliptic curve are 
essentially Jacobi sums \cite{LI:Mat410}.
They also satisfy the (local) Riemann Hypothesis $w\bar{w}=p$, being primes of a certain field extension of the rationals, 
splitting the rational prime $p$.
It is natural to expect that such Weil zeros can be interpreted in a natural way as periods in a geometric sense, of the corresponding elliptic curve.

These zeta functions, the congruence zeta function in finite characteristic and Riemann zeta function, formally correspond,
and the later plays the role of a partition function of the so called {\em Riemann gas}.
There are even more intriguing interpretations relating these topics with statistical mechanics 
(\cite{Julia}, and references therein).
This hints to a whole uncovered territory of understanding Physics in terms of Number Theory \cite{LI:Rem-PNT}.
The geometric and cohomological interpretation of the local factors will be addressed in a future article.

\subsection{A Second Amazing Connection?}
Since Veneziano amplitudes in String Theory are ``2D-versions'' of Feynman integrals, 
which are quite often linear combinations of MZVs (generalized Riemann Zeta Functions), it is natural to ponder:
is there a ``Second Amazing Connection'' , relating them to L-functions?

Now L-functions have Euler products, and via the local-to-global Principle, they have p-factors $L_p$,
which are generating functions for the number of points of elliptic curves in finite characteristic.
Then in what way, if any, are these local factors related to the local amplitudes represented by Jacobi sums?
Answering such questions is the next step of author's research program, 
investigating a discrete version of String Theory \cite{LI:DST}.

\section{Conclusions}
Periods are a new class of numbers obtained as algebraic integrals, i.e. integrals of rational/algebraic functions over 
rational/algebraic domains. 
A rich source of such periods is quantum physics, with its Feynman integrals representing amplitudes of probability of scattering experiments. 
Quite surprisingly, these numbers (amplitudes) turn out to be linear combinations of periods occurring in Number Theory, 
namely zeta and multiple zeta values, together with their generalizations (L-functions) \cite{QuantaMagazine}.

Since the prototypical framework exhibiting periods, e.g. Feynman diagrams, MZV and Chen integrals, Kontsevich integral (universal invariant), Kontsevich deformation quantization, just to name a few, involves what seems to be a homotopical de Rham theory 
(one of the interpretations of Chen Integrals), the connection between knots and primes will be instrumental \cite{Knots_and_Primes}.

Currently there is a surge of interest in understanding their structure, towards a theory of periods, organizing the topic in a similar way Class Field Theory organized the topic of algebraic numbers approximately a century ago. 
One future direction of study of periods consists in using the local-to-global principle, 
and study their local cohomological interpretation.

According to a general principle ({\em Local-to-global Principle}),
the local factors provide information about the global case,
of number fields.\footnote {One should include the factor corresponding to the ``prime at infinity''.}.
The ``local'' factors should have an interpretation as an ``integral'' in the corresponding characteristic. 
One future direction of development of the theory of periods is to investigate these local factors.

Note that Veneziano amplitude in String Theory can be seen as the 'local' representative for Feynman Integral, 
and that Jacobi sums are a discrete analog of Veneziano amplitudes.
On the other hand, it is not clear how Feynman amplitudes, which are rational linear combinations of MZVs,
relate to the later; it is natural to hope that the local-to-global principle mentioned above,
relating the ``continuum'' (completion of $Q$), and the discrete ($Z$ and p-adic completions), 
is a starting point for a study aiming to understand this ``unexpected / unexplained'' efficiency of Mathematics in Physics.

\section{Acknowledgements}
The first author would like to thank for the opportunity to advise the research project of the second author  on such an
interesting and modern topic, and to thank Professor Maxim Lvovich Kontsevich and Professor Yuri Ivanovich Manin for 
implicit or direct encouragements to {\em start} pursuing such a difficult subject.


\end{document}